\documentclass{article}
\usepackage{nips07submit_e,times}
\usepackage{amsmath,url,amsthm,amssymb,amsbsy,fancybox}
\usepackage{algorithmic,algorithm}
\usepackage{subfigure,epsfig,graphicx}
\def\argmin{\mathop{\text{arg\,min}}}
\def\argmax{\mathop{\text{arg\,max}}}

\let\hat\widehat

\newcommand{\by}{{\mathbf y}}

\newcommand{\bX}{{\mathbf X}}
\newcommand{\bY}{{\mathbf Y}}

\newcommand{\bb}{{\boldsymbol \beta}}
\newcommand{\ba}{{\boldsymbol \alpha}}

\newtheorem{lemma}{Lemma}

\newtheorem{theorem}{Theorem}
\newtheorem{remark}{Remark}

\title{A Smoothing Stochastic Gradient Method for Composite Optimization}

\author{
\begin{tabular}{c}
{Qihang Lin } \\[10pt]
\rm Tepper School of Business\\
\rm Carnegie Mellon University \\
\rm Pittsburgh, PA 15213
\end{tabular}
\begin{tabular}{c}
{Xi Chen } \\[10pt]
\rm School of Computer Science\\
\rm Carnegie Mellon University \\
\rm Pittsburgh, PA 15213
\end{tabular}
\begin{tabular}{c}
{Javier Pena } \\[10pt]
\rm Tepper School of Business\\
\rm Carnegie Mellon University \\
\rm Pittsburgh, PA 15213
\end{tabular}
}

\begin{document}
\maketitle

\begin{abstract}
We consider the unconstrained optimization problem whose objective function is composed of a smooth and a non-smooth conponents where the smooth component is the expectation a random function. This type of problem arises in some interesting applications in machine learning. We propose a stochastic gradient  descent algorithm for this class of optimization problem. When the non-smooth component has a particular structure, we propose another stochastic gradient  descent algorithm by incorporating a smoothing method into our first algorithm. The proofs of the convergence rates of these two algorithms are given and we show the numerical performance of our algorithm by applying them to regularized linear regression problems with different sets of synthetic data.  
\end{abstract}

\section{Introduction}
In the past decade, convex programming has been widely applied in a variety of areas including statistical estimation, machine learning, data mining and signal processing. One of the most popular classes of convex programming problems, which appears in many different applications such as lasso \cite{Tibshirani:96} and group lasso \cite{Yuan:06}, can be formulated as the following minimization problem, 
\begin{equation}
\label{obj}
\min_x\phi(x)\equiv f(x)+h(x).
\end{equation}
Here, the function $f(x)$ is smooth and convex and its gradient $\nabla f(x)$ is Lipschitz continuous with a Lipschitz constant $L$. The function $h(x)$ is assumed to be convex but non-smooth. 

Interior-point methods \cite{Boyd:03,SOCP:98} are considered as general algorithms for solving different types of convex programming. However, they are not scalable for problems with even moderate sizes due to the big cost of solving the Newton linear equations system in each main iteration. A block coordinate method was developed by Tseng and Yun \cite{Tseng:09} and applied to the problems which can be formulated by (\ref{obj}) in \cite{Lukas:08,Friedman:10}. However, this method requires a separable structure in the objective function that does not exist in some applications such as overlapped group lasso \cite{Bach:09}.

Recently, gradient descent methods, or so called first-order methods, e.g. \cite{Nesterov:03,Paul:08, Auslender:06}, have attracted great interest because they are not only relatively easy to implement but also capable of solving some challenging problems with huge size. For problems formulated by (\ref{obj}), the first-order methods proposed in \cite{Nesterov:03, Paul:08, Auslender:06} can achieve a $O(\frac{1}{N^2})$ convergence rate, where $N$ is the number of iterations. 
The different variations of gradient descent algorithm have been successively applied to different types of problems, for example, nuclear norm regularization \cite{Pong:10,Toh:09}, $\ell_1/\ell_2$-norm regularization \cite{Liu:09} and so on. 

In each loop of a gradient descent algorithm, a projection mapping, which itself is a minimization problem, must be solved in order to find the next intermediate solution. Although a projection mapping usually has a closed form solution which guarantees the efficiency of a gradient descent algorithm, there exists a class of problems formulated by (\ref{obj}) including overlapped group lasso \cite{Bach:09} and fused lasso \cite{Tibshirani:09}, for which a closed form solution for the projection mapping is not available. Fortunately, the non-smooth component $h(x)$ in the objective function usually has a particular type of structure of the form
\begin{equation}
\label{maxh}
h(x)=\max_{v\in Q}v^TAx,
\end{equation} 
where $Q$ is a compact set.
Nesterov \cite{Nesterov:05} proposed a scheme to construct a smooth approximation of the objective function $\phi(x)$ in (\ref{obj}) whenever $h(x)$ satisfies (\ref{maxh}) and also a gradient descent algorithm to minimize the approximated problem. 
This approximation scheme and algorithm have been applied to overlapped group lasso \cite{Chen:10a} and fused lasso \cite{Chen:10b} and given good numerical results.

A stochastic gradient descent algorithm can be considered as a gradient descent algorithm that utilizes random approximations of  gradients instead of exact gradients. During the past few years, a significant amount of work has been done to develop  stochastic gradient descent algorithms for different problems (see, e.g. \cite{Nemirovski:09,Lan:10a, Lan:10b, Hu:09}). One reason for people to consider stochastic gradient is that the exact gradients are computationally expensive or sometimes even impossible to evaluate. One typical case to consider stochastic gradients is stochastic optimization where the objective function $f(x)$ is given as an expectation of a random value function $F(x,\xi)$, i.e. $f(x)=\mathbb{E}F(x,\xi)$, where $\xi$ is a random variable and the expectation is taken over $\xi$. In this situation, a multidimensional numerical integral would be needed to compute the exact gradient $\nabla f(x)= \mathbb{E}\nabla F(x,\xi)$. That would be too time consuming especially when the dimension is high. It could be even worse if the distribution of $\xi$ is unknown so that there is no way to get the exact gradient  $\nabla f(x)$.

In this paper, we consider the optimization problem formulated by (\ref{obj}) but we further assume the smooth component of the objective function to be of the form $f(x)=\mathbb{E}F(x,\xi)$ just as it is the case in a stochastic optimization problem. Hence, we have to consider the stochastic gradient for the reasons we mentioned above. We propose a stochastic gradient descent algorithm to solve (\ref{obj}) under the assumption that a stochastic approximation of $\nabla f(x)$, denoted by $G(x,\xi)$, is available in each iteration, where $\xi$ is a random variable. We also assume $G(x,\xi)$ is an unbiased estimate of $\nabla f(x)$, i.e., $\mathbb{\mathbb{E}}G(x,\xi)=\nabla f(x)$, and $\mathbb{\mathbb{E}}\|\nabla f(x)-G(x,\xi)\|^2\leq\sigma^2$ for some nonnegative constant $\sigma$.\footnote{In this paper, the notation $\|\cdot\|$ without any subscript presents the Euclidean norm of a vector.} We show that our stochastic gradient algorithm obtain a convergence rate of $O(\frac{1}{\sqrt{N}})$, which is the same, up to a constant independent of $N$, as the convergence rates showed in \cite{Nemirovski:09, Lan:10a, Lan:10b, Hu:09} without assuming strongly convexity for the objective functions.

Our algorithms can viewed as an extension of the Algorithm 1 in \cite{Paul:08} by utilizing stochastic gradients. The choices of the parameters $\gamma_t$ and $\gamma^*$ in our algorithms are inspired by the choices of similar parameters in \cite{Lan:10a}. But our algorithm is different from Lan's accelerated stochastic approximation method in \cite{Lan:10a} in that they assumed $G(x,\xi)$ is a stochastic subgradient for the non-smooth objective function $\phi(x)$ in (\ref{obj}) while we assume $G(x,\xi)$ is a stochastic gradient only for the smooth component $f(x)$ in (\ref{obj}). Although our method is similar to a simplified version of AC-SA algorithm proposed in \cite{Lan:10b}, we use a different choice of parameters and focus more on the effect of smoothing technique in this paper. 

Similar to the exact gradient descent algorithm, the existing stochastic gradient descent algorithms, e.g. \cite{Lan:10b, Hu:09}, have to solve a projection mapping in each iteration, which may not have a closed form solution. However, when the function $h(x)$ in (\ref{obj}) has the structure given in (\ref{maxh}), we propose
another stochastic gradient descent algorithm by incorporating the smoothing technique proposed by Nesterov \cite{Nesterov:05}. This method replaces $h(x)$ by its smooth approximation such that the projection mapping always obtains a closed form solution. Hence, our method can be applied to problems like overlapped group lasso and fused lasso, which other stochastic gradient descent algorithm can not solve efficiently.

According to \cite{Nesterov:05,Lan:09}, the convergence rates of the accelerated gradient algorithms will be reduced from $O(\frac{1}{N^2})$ to $O(\frac{1}{N})$ if the smoothing technique in \cite{Nesterov:05} is applied. However, we show that the convergence rate for our stochastic gradient algorithm remains $O(\frac{1}{\sqrt{N}})$ even when the smoothing technique is applied. In other word, although the price of the smoothing technique is kind of high for deterministic gradient methods, it is totally free for stochastic gradient methods.

The rest of this paper is organized as follows: in the next section, we present our stochastic gradient descent algorithm
and prove its convergence rate. Combining our first algorithm with a smoothing technique, we propose another stochastic gradient descent algorithm in Section 3. In Section 4, we show the numerical results on simulated data, followed by a concluding section at the end.

\section{Stochastic Gradient Descent Algorithm}

In this section, we propose a stochastic gradient algorithm summerized in Algorithm \ref{algo:sgsmooth} below to solve the following optimization problem:
\begin{equation}
\label{sobj}
\min_x\phi(x)\equiv f(x)+h(x)=\mathbb{E}F(x,\xi)+h(x).
\end{equation}
where the expectation $\mathbb{E}$ is taken over the random variable $\xi$.
We assume that at every point $x$, there is a random vector $G(x,\xi)$ determined by $x$ and $\xi$ such that $\mathbb{E} G(x,\xi)=\nabla f(x)$ and $\mathbb{\mathbb{E}}\|\nabla f(x)-G(x,\xi)\|^2\leq\sigma^2$ for any $x$. In the $t$th step of Algorithm \ref{algo:sgsmooth}, we independently generate a random variable $\xi_t$ from the distribution of $\xi$ and compute $G(y_t,\xi_t)$ at a point $y_t$ based on $\xi_t$.

\begin{algorithm}[!th]
\caption{Stochastic Gradient Descent Algorithm (SG)}

\textbf{Input}: The total number of iterations $N$ and the Lipschitz constant $L$.

\textbf{Initialization}:  Choose $\theta_t=\frac{2}{2+t}$ and $\gamma_t=\frac{2}{t+2}(\frac{N^{\frac{3}{2}}}{L}+2)$. Set $x_0=0$, $z_0=0$ and $t=0$.

\textbf{Iterate} for $t=0,1,2,\ldots N$:
\begin{enumerate}
\item $y_t=(1-\theta_t)x_t+\theta_tz_t$
\item Generate a random vector $\xi_t$ independently from the distribution of $\xi$
\item $z_{t+1}=\argmin_x\{\left\langle x,G(y_t,\xi_t)\right\rangle +\frac{\gamma_t L}{2}\|x-z_t\|^2+h(x)\}$
\item $x_{t+1}=(1-\theta_t)x_t+\theta_tz_{t+1}$
\end{enumerate}

\textbf{Output}: $x_{N+1}$

\label{algo:sgsmooth}
\end{algorithm}

\begin{remark}
Algorithm \ref{algo:sgsmooth} is based on the Algorithm 1 proposed by Tseng in \cite{Paul:08}. Algorithm \ref{algo:sgsmooth} is different from Tseng's algorithm in two aspects. First, the exact gradient used in Tseng's algorithm is replaced by the stochastic gradient due to the difficulty of computing the exact gradient in our problems as mentioned in Section 1. Second, two sequences of step lengths, $\theta_t$ and $\gamma_t$, are maintained to guarantee the convergence of our algorithm while in Tseng's algorithm, one sequence of step length $\{\theta_t\}$ is enough. It should be pointed out that if we set $\gamma_t$ in Algorithm \ref{algo:sgsmooth} to be $\frac{2\gamma^*}{L(t+1)}$ with $\gamma^*=\max\left\{2L,\left[\frac{2\sigma^2N(N+1)(N+2)}{3\|x_0-x^*\|^2}\right]^{\frac{1}{2}}\right\}$, Algorithm 1 just becomes the AC-SA algorithm proposed by Ghadimi and Lan \cite{Lan:10b} for unconstrained optimization problems when $f(x)$ is just convex but not necessarily strongly convex. However, the parameter $\gamma^*$ in the AC-SA algorithm is hard to evaluate because it depends on the optimal solution $x^*$ and $\sigma$ while the parameters in our algorithm are relatively simple and  result in better numerical performances as shown in Section 4.
\end{remark}

Here, we assume the projection mapping in step 3 in Algorithm \ref{algo:sgsmooth} can be solved efficiently or has a closed form solution. This is true in many problems where the non-smooth term $h(x)$ is  $\ell_1$-norm \cite{Tibshirani:96}, $\ell_1/\ell_2$-norm  \cite{Liu:09,Yuan:06,Jacob:09} or nuclear norm of $x$ \cite{Pong:10,Toh:09}.

Using the same notations in Algorithm \ref{algo:sgsmooth}, we present the convergence rate of Algorithm \ref{algo:sgsmooth} in the following theorem. Some techniques in the proof are inspired by the proofs for the complexity results in \cite{Paul:08} and \cite{Lan:10a}.
\begin{theorem}
\label{smoothconveg}
Suppose $N$ is the total number of iterations in Algorithm \ref{algo:sgsmooth} and $x^*$ is the optimal solution of (\ref{sobj}) and we assume that the stochastic gradient $G(x,\xi)$ satisfies $\mathbb{\mathbb{E}}\|\nabla f(x)-G(x,\xi)\|^2\leq\sigma^2$ for all $x$. Then we have
\begin{equation*}
\mathbb{E}\left(\phi(x_{N+1})-\phi(x^*)\right)\leq\frac{2D^2+\sigma^2}{(N+2)^{\frac{1}{2}}}+L\frac{4D^2+2\sigma^2}{(N+2)^2},
\end{equation*}
where  $D=\|x^*-z_0\|$.
\end{theorem}

Three technical lemmas are presented here before the proof of the convergence rate of Algorithm \ref{algo:sgsmooth} is given. Lemma \ref{step} is an inequality satisfied by the step lengths we choose in Algorithm \ref{algo:sgsmooth} and Lemma \ref{ineq} is a basic property of convex functions. Lemma \ref{minzer}, shown in \cite{Lan:10a}, is a technical result used to characterize the optimal solution of the projection mapping in step 3 of Algorithm 1. We put the proof of Lemma 3 here just for the completeness.

\begin{lemma}
\label{step}
Suppose the sequences $\{\theta_{t}\}$ and $\{\gamma_{t}\}$ are chosen as in Algorithm \ref{algo:sgsmooth}, we have
$\frac{1-\theta_{t+1}}{\theta_{t+1}\gamma_{t+1}}\leq \frac{1}{\theta_{t}\gamma_{t}}$ and $\gamma_t>\theta_t$. 
\end{lemma}
\begin{proof}
The first inequality comes from
\begin{equation*}
\frac{\theta_{t}(1-\theta_{t+1})}{\theta_{t+1}}=\frac{\frac{2}{t+2}(1-\frac{2}{t+3})}{\frac{2}{t+3}}=\frac{t+1}{t+2}\leq \frac{t+2}{t+3}=\frac{\gamma_{t+1}}{\gamma_{t}}.\\
\end{equation*}
The second one is obvious.
\end{proof}

\begin{lemma}
\label{ineq}
If $f(x)$ is a smooth convex function on $\mathbb{R}^n$ and $\nabla f(x)$ is Lipschitz continuous with a Lipschitz constant $L$, then we have
\begin{equation*}
f(y)+\left\langle x-y,\nabla f(y)\right\rangle\leq f(x) \leq f(y)+\left\langle x-y,\nabla f(y)\right\rangle +\frac{L}{2}\|x-y\|^2
\end{equation*}
for all $x,y$.
\end{lemma}
This is a classical property of convex functions. For a proof, see \cite{Conv:01}.

\begin{lemma}
\label{minzer}(See also \cite{Paul:08}, \cite{Lan:10a} and \cite{Lan:10b}) 
Suppose $\psi(x)$ is convex and $z^*$ is the optimal solution of $\min_z\psi(z)+\frac{1}{2}\|z-\hat{z}\|^2$, then we have the following inequality:
\begin{equation*}
\psi(z^*)+\frac{1}{2}\|z^*-\hat{z}\|^2\leq \psi(x)+\frac{1}{2}\|x-\hat{z}\|^2-\frac{1}{2}\|x-z^*\|^2
\end{equation*}
for all $x$.
\end{lemma}
\begin{proof}
The definition of $z^*$ implies that there exists a subgradient $\eta$ in $\partial \psi(z^*)$, the subdifferential of 
function $\psi(z)$ at $z^*$, such that

\begin{equation}
\label{alemma3}
\left\langle \eta+z^*-\hat{z},x-z^*\right\rangle\geq 0\mbox{ for all } x.
\end{equation}
And the convexity of $\psi(x)$ implies
\begin{equation}
\label{blemma3}
\psi(x)
\geq\psi(z^*)+\left\langle \eta,x-z^*\right\rangle \mbox{ for all } x.
\end{equation}
It is easy to verify that
\begin{equation}
\label{clemma3}
\frac{1}{2}\|\hat{z}-x\|^2=\frac{1}{2}\|\hat{z}-z^*\|^2+\left\langle z^*-\hat{z},x-z^*\right\rangle+\frac{1}{2}\|z^*-x\|^2 \mbox{ for all } x.
\end{equation}
Using the (\ref{alemma3})(\ref{blemma3})(\ref{clemma3}) above, we conclude that
\begin{eqnarray*}
\psi(x)+\frac{1}{2}\|x-\hat{z}\|^2&=&\psi(x)+\frac{1}{2}\|\hat{z}-z^*\|^2+\left\langle z^*-\hat{z},x-z^*\right\rangle+\frac{1}{2}\|z^*-x\|^2\\
&\geq&\psi(z^*)+\frac{1}{2}\|\hat{z}-z^*\|^2+\left\langle \eta+z^*-\hat{z},x-z^*\right\rangle+\frac{1}{2}\|z^*-x\|^2\\
&\geq&\psi(z^*)+\frac{1}{2}\|\hat{z}-z^*\|^2+\frac{1}{2}\|z^*-x\|^2\mbox{ for all } x.\\
\end{eqnarray*}
\end{proof}

Here, we give the proof of Theorem \ref{smoothconveg}.
\begin{proof}
We define $\Delta_t=\nabla f(y_t)-G(y_t,\xi_t)$ so that $\mathbb{E}\Delta_t=0$ and $\mathbb{\mathbb{E}}\|\Delta_t\|^2\leq\sigma^2$. We can bound $\phi(x_{t+1})$ from above as follows.
\begin{eqnarray}\nonumber
\phi(x_{t+1})&=&f(x_{t+1})+h(x_{t+1})\\\nonumber
&\leq& f(y_t)+\left\langle x_{t+1}-y_t, \nabla f(y_t)\right\rangle+\frac{L}{2}\|x_{t+1}-y_{t}\|^2+h(x_{t+1})\\\nonumber
&\leq&(1-\theta_t)(f(y_{t})+\left\langle x_{t}-y_t, \nabla f(y_t)\right\rangle+h(x_t))+\\\nonumber
&&\theta_t(f(y_{t})+\left\langle z_{t+1}-y_t, \nabla f(y_t)\right\rangle+h(z_{t+1}))+\theta_t^2\frac{L}{2}\|z_{t+1}-z_{t}\|^2\\\nonumber
&\leq&(1-\theta_t)\phi(x_{t})+\theta_t(f(y_{t})+\left\langle z_{t+1}-y_t, \nabla f(y_t)\right\rangle+h(z_{t+1}))\\\nonumber
&&+\theta_t^2\frac{L}{2}\|z_{t+1}-z_{t}\|^2\\\nonumber
&=&(1-\theta_t)\phi(x_{t})+\theta_t(f(y_{t})+\left\langle z_{t+1}-y_t, G(y_t,\xi_t)\right\rangle+h(z_{t+1})+\gamma_t\frac{L}{2}\|z_{t+1}-z_{t}\|^2)\\\label{ineq1}
&&+(\theta_t^2-\theta_t\gamma_t)\frac{L}{2}\|z_{t+1}-z_{t}\|^2+\theta_t\left\langle z_{t+1}-y_t,\Delta_t\right\rangle\\\nonumber
\end{eqnarray}
The first and third inequalities above are due to Lemma \ref{ineq} and the second one is implied by the updating equations for $y_t$ and $x_{t+1}$ and the convexity of $h(x)$.

According to Lemma \ref{minzer} with $\psi(z)=\frac{1}{\gamma_t L}(\left\langle z,G(y_t,\xi_t)\right\rangle+h(z))$, $z^*=z_{t+1}$ and $\hat{z}=z_t$, we get 
\begin{eqnarray}
\label{applylemma}
&&(\left\langle z_{t+1},G(y_t,\xi_t)\right\rangle+h(z_{t+1}))+\frac{\gamma_t L}{2}\|z_{t+1}-z_t\|^2\\\nonumber
&\leq& (\left\langle x,G(y_t,\xi_t)\right\rangle+h(x))+\frac{\gamma_t L}{2}\|x-z_t\|^2-\frac{\gamma_t L}{2}\|x-z_{t+1}\|^2 \mbox{ for all } x.
\end{eqnarray}

By choosing $x=x^*$ in (\ref{applylemma}), it follows from (\ref{ineq1})  and (\ref{applylemma}) that 
\begin{eqnarray}\nonumber
\phi(x_{t+1})&\leq&(1-\theta_t)\phi(x_{t})+\theta_t(f(y_{t})+\left\langle x^*-y_t,  G(y_t,\xi_t)\right\rangle+h(x^*)+\gamma_t\frac{L}{2}\|x^*-z_{t}\|^2)\\\nonumber
&&-\theta_t\gamma_t\frac{L}{2}\|x^*-z_{t+1}\|^2+(\theta_t^2-\theta_t\gamma_t)\frac{L}{2}\|z_{t+1}-z_{t}\|^2+\theta_t\left\langle z_{t+1}-y_t,\Delta_t\right\rangle\\\nonumber
&=&(1-\theta_t)\phi(x_{t})+\theta_t(f(y_{t})+\left\langle x^*-y_t,  \nabla f(y_t)\right\rangle+h(x^*)+\gamma_t\frac{L}{2}\|x^*-z_{t}\|^2)\\\nonumber
&&-\theta_t\gamma_t\frac{L}{2}\|x^*-z_{t+1}\|^2+(\theta_t^2-\theta_t\gamma_t)\frac{L}{2}\|z_{t+1}-z_{t}\|^2+\theta_t\left\langle z_{t+1}- x^*,\Delta_t\right\rangle.\\\label{tempineq}
\end{eqnarray}
Here, the equality above holds because $\Delta_t=\nabla f(y_t)-G(y_t,\xi_t)$.

By Lemma \ref{ineq}, the term $f(y_{t})+\left\langle x^*-y_t,  \nabla f(y_t)\right\rangle+h(x^*)$ in (\ref{tempineq}) is no more than $\phi(x^*)$. Hence, we can upper bound $\phi(x_{t+1})$ as:
\begin{eqnarray*}
\phi(x_{t+1})&\leq&(1-\theta_t)\phi(x_{t})+\theta_t\phi(x^*)+\theta_t\gamma_t\frac{L}{2}\|x^*-z_{t}\|^2-\theta_t\gamma_t\frac{L}{2}\|x^*-z_{t+1}\|^2\\
&&-(\theta_t\gamma_t-\theta_t^2)\frac{L}{2}\|z_{t+1}-z_{t}\|^2+\theta_t\left\langle z_{t+1}-y_t,\Delta_t\right\rangle-\theta_t\left\langle x^*-y_t,\Delta_t\right\rangle\\
&=&(1-\theta_t)\phi(x_{t})+\theta_t\phi(x^*)+\theta_t\gamma_t\frac{L}{2}\|x^*-z_{t}\|^2-\theta_t\gamma_t\frac{L}{2}\|x^*-z_{t+1}\|^2\\
&&-(\theta_t\gamma_t-\theta_t^2)\frac{L}{2}\|z_{t+1}-z_{t}\|^2+\theta_t\left\langle z_{t+1}-z_{t},\Delta_t\right\rangle+\theta_t\left\langle z_{t}-x^*,\Delta_t\right\rangle\\
&\leq&(1-\theta_t)\phi(x_{t})+\theta_t\phi(x^*)+\theta_t\gamma_t\frac{L}{2}\|x^*-z_{t}\|^2-\theta_t\gamma_t\frac{L}{2}\|x^*-z_{t+1}\|^2\\
&&-(\theta_t\gamma_t-\theta_t^2)\frac{L}{2}\|z_{t+1}-z_{t}\|^2+\theta_t \|z_{t+1}-z_{t}\|\|\Delta_t\|+\theta_t\left\langle z_{t}-x^*,\Delta_t\right\rangle\\
&\leq&(1-\theta_t)\phi(x_{t})+\theta_t\phi(x^*)+\theta_t\gamma_t\frac{L}{2}\|x^*-z_{t}\|^2-\theta_t\gamma_t\frac{L}{2}\|x^*-z_{t+1}\|^2\\
&&+\frac{\theta_t\|\Delta_t\|^2}{2L(\gamma_t-\theta_t)}+\theta_t\left\langle z_{t}-x^*,\Delta_t\right\rangle.\\
\end{eqnarray*}
We get the second inequality above by applying Cauchy-Schwarz inequality to $\left\langle z_{t+1}-z_{t},\Delta_t\right\rangle$ and the last inequality comes from applying the inequality $-ax^2+bx\leq\frac{b^2}{4a}$ with $a>0$ to $a=(\theta_t\gamma_t-\theta_t^2)$, $x=\|z_{t+1}-z_{t}\|$ and $b=\theta_t \|\nabla_t\|$. Note that $(\theta_t\gamma_t-\theta_t^2)>0$ from Lemma 1.

Until now, we have already got
\begin{eqnarray}
\label{tempineq2}
\phi(x_{t+1})&\leq&(1-\theta_t)\phi(x_{t})+\theta_t\phi(x^*)+\theta_t\gamma_t\frac{L}{2}\|x^*-z_{t}\|^2-\theta_t\gamma_t\frac{L}{2}\|x^*-z_{t+1}\|^2\\\nonumber
&&+\frac{\theta_t\|\Delta_t\|^2}{2L(\gamma_t-\theta_t)}+\theta_t\left\langle z_{t}-x^*,\Delta_t\right\rangle.\\\nonumber
\end{eqnarray}

We define $\mathbb{E}(\left\langle z_t-x^*,\Delta_t\right\rangle|\xi_1,\dots,\xi_{t-1})$  to be the conditional expectation of $\left\langle z_t-x^*,\Delta_t\right\rangle$ under the condition that $\xi_1,\dots,\xi_{t-1}$ have been generated. According to Algorithm \ref{algo:sgsmooth}, $z_t$ is only determined by $\xi_1,\dots,\xi_{t-1}$ but not by $\xi_t$. Hence,  $\mathbb{E}(\left\langle z_t-x^*,\Delta_t\right\rangle|\xi_1,\dots,\xi_{t-1})=0$ because $\mathbb{E}\Delta_t=0$. By the iterative property of expectation, we have 
\begin{equation*}
\mathbb{E}\left\langle z_t-x^*,\Delta_t\right\rangle=\mathbb{E}(\mathbb{E}(\left\langle z_t-x^*,\Delta_t\right\rangle|\xi_1,\dots,\xi_{t-1}))=\mathbb{E}0=0.
\end{equation*}

Hence, if we subtract $\phi(x*)$ from both sides of inequality (\ref{tempineq2}) and take the expectation, we will have
\begin{eqnarray}
\label{tempineq3}
&&\mathbb{E}(\phi(x_{t+1})-\phi(x^*))\\\nonumber
&\leq&(1-\theta_t)(\mathbb{E}(\phi(x_{t}))-\phi(x^*))+\theta_t\gamma_t\frac{L}{2}\mathbb{E}\|x^*-z_{t}\|^2-\theta_t\gamma_t\frac{L}{2}\mathbb{E}\|x^*-z_{t+1}\|^2\\\nonumber
&&+\frac{\theta_t\sigma^2}{2L(\gamma_t-\theta_t)}.\\\nonumber
\end{eqnarray}
Moreover, we divide both sides of inequality (\ref{tempineq3}) by $\theta_t\gamma_t$ and get
\begin{eqnarray}
\label{tempineq1}
&&\frac{1}{\theta_t\gamma_t}(\mathbb{E}(\phi(x_{t+1}))-\phi(x^*))\\\nonumber
&\leq&\frac{1-\theta_t}{\theta_t\gamma_t}(\mathbb{E}(\phi(x_{t}))-\phi(x^*))+\frac{L}{2}\mathbb{E}\|x^*-z_{t}\|^2-\frac{L}{2}\mathbb{E}\|x^*-z_{t+1}\|^2+\frac{\sigma^2}{2L\gamma_t(\gamma_t-\theta_t)}\\\nonumber
&\leq&\frac{1}{\theta_{t-1}\gamma_{t-1}}(\mathbb{E}(\phi(x_{t}))-\phi(x^*))+\frac{L}{2}\mathbb{E}\|x^*-z_{t}\|^2-\frac{L}{2}\mathbb{E}\|x^*-z_{t+1}\|^2+\frac{\sigma^2}{2L\gamma_t(\gamma_t-\theta_t)},\\\nonumber
\end{eqnarray}
where the second inequality comes form Lemma \ref{step}.

By applying inequality (\ref{tempineq1}) recursively, we obtain
\begin{equation}
\label{ineqrecursive}
\frac{1}{\theta_N\gamma_N}(\mathbb{E}(\phi(x_{N+1}))-\phi(x^*))\leq\frac{L}{2}\mathbb{E}\|x^*-z_{0}\|^2+\frac{\sigma^2}{2L}\sum_{t=0}^N\frac{1}{\gamma_t(\gamma_t-\theta_t)}.\\
\end{equation}

The definitions of $\theta_t$ and $\gamma_t$ imply
\begin{eqnarray*}
&&\sum_{t=0}^N\frac{1}{\gamma_t(\gamma_t-\theta_t)}=\sum_{t=0}^N\frac{(t+2)^2}{4(N^{\frac{3}{2}}/L+2)(N^{\frac{3}{2}}/L+1)}
\leq\frac{(N+2)(N+3)(2N+5)L^2}{24N^3}\\
&\leq&\frac{12N^3L^2}{24N^3}=\frac{L^2}{2},\\
\end{eqnarray*}
which, together with inequality (\ref{ineqrecursive}), implies:
\begin{eqnarray*}
&&\mathbb{E}(\phi(x_{N+1}))-\phi(x^*)\\
&\leq&\theta_N\gamma_N(\frac{L}{2}D^2+\frac{\sigma^2L}{4})\\
&\leq&\frac{4}{(N+2)^2}\frac{N^{\frac{3}{2}}}{L}(\frac{L}{2}D^2+\frac{\sigma^2L}{4})+\frac{8}{(N+2)^2}(\frac{L}{2}D^2+\frac{\sigma^2L}{4})\\
&\leq&\frac{2D^2+\sigma^2}{(N+2)^{\frac{1}{2}}}+L\frac{4D^2+2\sigma^2}{(N+2)^2}.
\end{eqnarray*}
\end{proof}

\begin{remark}
Algorithm \ref{algo:sgsmooth} obtains an asymptotically rate of convergence $\mathbb{E}\left(\phi(x_{N+1})-\phi(x^*)\right)=O(\frac{1}{\sqrt{N}})$ which is the same as the convergence rate of  the AC-SA algorithm proposed by Ghadimi and Lan \cite{Lan:10b} up to a constant factor. This convergence rate is also known to be asymptotically optimal (see \cite{Nemirovski:83}) in terms of the number of iterations $N$.
\end{remark}

\section{Smoothing Stochastic Gradient Descent Algorithm}
Notice that a projection mapping 
\begin{equation}
\label{projmap}
z_{t+1}=\argmin_x\{\left\langle x,G(y_t,\xi_t)\right\rangle +\frac{\gamma_t L}{2}\|x-z_t\|^2+h(x)\}
\end{equation}
must be solved in the step 3 of Algorithm \ref{algo:sgsmooth}. Similar type of projection mappings also appear in other gradient or stochastic gradient algorithms. As indicated in Section 1, (\ref{projmap}) does not necessarily have a closed from solution. This happens, in particular,  in group lasso problem with overlapped group structures \cite{Bach:09} and fused lasso \cite{Tibshirani:09}. In this case, another iterative algorithm has to be designed for solving this projection mapping in each iteration of Algorithm \ref{algo:sgsmooth}, which could make Algorithm \ref{algo:sgsmooth} very slow for practical applications.

In order to modify Algorithm \ref{algo:sgsmooth} the problems whose corresponding projection mappings have no closed form, we utilize the smoothing technique proposed by Nesterov \cite{Nesterov:05} to construct a smooth approximation for problem (\ref{sobj}) before we apply Algorithm \ref{algo:sgsmooth}.

Suppose the non-smooth part $h(x)$ in (\ref{sobj}) can be represented as
\begin{equation*}
h(x)=\max_{v\in Q}{v^TAx},
\end{equation*}
we consider the function
\begin{equation}
\label{hu}
h_{\mu}(x)=\max_{v\in Q}\{{v^TAx}-\mu d(v)\}.
\end{equation}
Here, the parameter $\mu$ is a positive constant and $d(v)$ is a smooth and strongly convex function on $Q$. According to \cite{Nesterov:05}, the function $h_{\mu}(x)$ is a smooth lower approximation for $h(x)$ if $\mu$ is positive. In fact, it can be shown that 
\begin{equation*}
h_{\mu}(x)\leq h(x)\leq h_{\mu}(x)+\mu M \mbox{ for all } x, 
\end{equation*}
where $M=\max_{v\in Q}d(v)$.
Hence, the parameter $\mu$ controls the accuracy of approximation.

We denote by $v_{\mu}(x)$ the optimal solution of the maximization problem involved in (\ref{hu}). Since $d(v)$ is strongly convex, $v_{\mu}(x)$ is well-defined because of the uniqueness of the optimal solution. It is proved in \cite{Nesterov:05} that $h_{\mu}(x)$ is a smooth function whose gradient is
\begin{equation}
\label{ghmu}
\nabla h_{\mu}(x)=A^Tv_{\mu}(x).
\end{equation}

Therefore, the function $\phi_{\mu}(x)\equiv f(x)+ h_{\mu}(x)$ performs as a smooth lower approximation for $\phi(x)$ in problem (\ref{sobj}) and we have 
\begin{equation}
\label{mubound}
\phi_{\mu}(x)\leq \phi(x)\leq \phi_{\mu}(x)+\mu M \mbox{ for all } x.
\end{equation} 
By (\ref{ghmu}), the gradient of  $\phi_{\mu}(x)$ is 
\begin{equation}
\label{gphimu}
\nabla\phi_{\mu}(x)=\nabla f(x)+A^Tv_{\mu}(x).
\end{equation}
and $G(x,\xi)+A^Tv_{\mu}(x)$ provides its stochastic approximation. It is easy to see that $\mathbb{E} (G(x,\xi)+A^Tv_{\mu}(x))=\nabla\phi_{\mu}(x)$ and $\mathbb{\mathbb{E}}\|\nabla\phi_{\mu}(x)-G(x,\xi)-A^Tv_{\mu}(x)\|^2\leq\sigma^2$ for any $x$.

It is also shown in \cite{Nesterov:05} that the gradient $\nabla\phi_{\mu}(x)$ is Lipschitz-continuous with a Lipschitz constant
\begin{equation}
\label{Lcsmooth}
L_{\mu}=L+\frac{1}{c\mu}\|A\|^2,
\end{equation}
where $\|A\|=\max_{\|x\|=1,\|y\|=1}y^TAx$ and  $c>0$ is the strong convexity parameter of function $d(v)$.

Since $\phi_{\mu}(x)$ is a smooth function with a stochastic gradient $G(x,\xi)+A^Tv_{\mu}(x)$ at each $x$, we can apply Algorithm \ref{algo:sgsmooth} to minimize $\phi_{\mu}(x)$.  When the smooth parameter $\mu$ is small enough, the solution we get will also be a good approximate solution for (\ref{sobj}). This modified algorithm is proposed as Algorithm \ref{algo:sgsmoothappx} as follows.

\begin{algorithm}[!th]
\caption{Smoothing Stochastic Gradient Descent Algorithm (SSG)}

\textbf{Input}: The total number of iterations $N$, the Lipschitz constant $L$ for $f(x)$ and the smooth parameter $\mu$.

\textbf{Initialization}: Compute the Lipschitz constant $L_{\mu}$ by (\ref{Lcsmooth}).
 Choose $\theta_t=\frac{2}{2+t}$ and $\gamma_t=\frac{2}{t+2}(\frac{N^{\frac{3}{2}}}{L_{\mu}}+2)$. Set $x_0=0$, $z_0=0$ and $t=0$.

\textbf{Iterate} for $t=0,1,2,\ldots N$:
\begin{enumerate}
\item $y_t=(1-\theta_t)x_t+\theta_tz_t$
\item $v_{\mu}(y_t)=\argmax_{v\in Q}\{{v^TAy_t}-\mu d(v)\}$
\item Generate a random vector $\xi_t$ independently from the distribution of $\xi$
\item $z_{t+1}=\argmin_x\{\left\langle x,G(y_t,\xi_t)+A^Tv_{\mu}(y_t)\right\rangle +\frac{\gamma_t L_{\mu}}{2}\|x-z_t\|^2\}$
\item $x_{t+1}=(1-\theta_t)x_t+\theta_tz_{t+1}$
\end{enumerate}

\textbf{Output}: $x_{N+1}$

\label{algo:sgsmoothappx}
\end{algorithm}

\begin{remark}
Similar to the step 3 in Algorithm \ref{algo:sgsmooth}, Algorithm \ref{algo:sgsmoothappx} also has to solve a projection mapping in step 4. However, since $\phi_{\mu}(x)$ does not contain a non-smooth term like $h(x)$ in (\ref{sobj}), the projection mapping in step 4 is simply an unconstrained quadratic programming whose optimal solution has a closed form.
\end{remark}

Since Algorithm \ref{algo:sgsmoothappx} just solves an approximation of (\ref{sobj}), we have to make the smooth parameter $\mu$  small enough in order to make the solution returned by Algorithm \ref{algo:sgsmoothappx} a near-optimal one for (\ref{sobj}). However,  according to (\ref{Lcsmooth}) and  Theorem \ref{smoothconveg}, 
decreasing smooth parameter $\mu$ will increase the Lipschitz constant $L_{\mu}$ and more iterations will be needed in Algorithm \ref{algo:sgsmoothappx} in order to minimize $\phi_{\mu}(x)$. Fortunately, by Theorem \ref{smoothconveg}, the Lipschitz constant $L_{\mu}$ only appears in the $O(\frac{L_{\mu}}{N^2})$ component of the convergence rate, which is dominated by the $O(\frac{1}{N^{1/2}})$ component. This means that, as long as $\mu=O(\frac{1}{N^{\gamma}})$ with $\gamma\leq\frac{3}{2}$ ,which implies $L_{\mu}=O(N^{\gamma})$ with $\gamma\leq\frac{3}{2}$, the convergence rate of Algorithm \ref{algo:sgsmoothappx} is still $O(\frac{1}{N^{1/2}})$. Based on this observation, we prove the following convergence result for Algorithm \ref{algo:sgsmoothappx}  when $\mu=O(\frac{1}{N})$. The similar results can be found in \cite{Nesterov:05} and \cite{Lan:09}.

\begin{theorem}
\label{nonsmoothconveg} If we set $\mu=\frac{\|A\|}{(N+2)}$ in Algorithm \ref{algo:sgsmoothappx}, then after $N$ iterations, we will have:
\begin{equation*}
\mathbb{E}\phi(x_{N+1})-\phi(x^*)\leq\frac{2D^2+\sigma^2}{(N+2)^{\frac{1}{2}}}+L\frac{4D^2+2\sigma^2}{(N+2)^2}+\frac{\|A\|}{(N+2)}(M+\frac{4D^2+2\sigma^2}{c})
\end{equation*}
\end{theorem}

\begin{proof}
Because $\phi(x_{N+1})-\phi_{\mu}(x_{N+1})\leq\mu M$ and $\phi_{\mu}(x^*)-\phi(x^*)\leq 0$, we have
\begin{eqnarray}
&&\mathbb{E}\phi(x_{N+1})-\phi(x^*)\\\nonumber
&=&\mathbb{E}\phi(x_{N+1})-\mathbb{E}\phi_{\mu}(x_{N+1})+\mathbb{E}\phi_{\mu}(x_{N+1})-\phi_{\mu}(x^*)+\phi_{\mu}(x^*)-\phi(x^*)\\\nonumber
&\leq&\mu M+\mathbb{E}\phi_{\mu}(x_{N+1})-\phi_{\mu}(x^*)\\\nonumber
&\leq&\mu M+\frac{2D^2+\sigma^2}{(N+2)^{\frac{1}{2}}}+(L+\frac{1}{c\mu}\|A\|^2)\frac{4D^2+2\sigma^2}{(N+2)^2}.\\\nonumber
\end{eqnarray}
where the last inequality is by Theorem \ref{smoothconveg} and (\ref{Lcsmooth}).
Setting $\mu=\frac{\|A\|}{(N+2)}$, we have 
\begin{equation*}
\mathbb{E}\phi(x_{N+1})-\phi(x^*)\leq \frac{2D^2+\sigma^2}{(N+2)^{\frac{1}{2}}}+L\frac{4D^2+2\sigma^2}{(N+2)^2}+\frac{\|A\|}{(N+2)}(M+\frac{4D^2+2\sigma^2}{c})
\end{equation*}
\end{proof}

\begin{remark}
This theorem shows a difference between exact gradient descent and stochastic descent algorithm when smoothing technique is applied. The gradient descent algorithm proposed in \cite{Nesterov:05} obtains a convergence rate of $O(\frac{1}{N^2})$ but it has to be reduced to  $O(\frac{1}{N})$ after applying the smoothing technique. However, for Algorithm \ref{algo:sgsmoothappx}, smoothing technique only slows down a non-dominating component in the convergence rate such that Algorithm \ref{algo:sgsmoothappx} still obtains a convergence rate of $O(\frac{1}{N^{1/2}})$ which is the same as Algorithm \ref{algo:sgsmooth}.  In other words, the price paid for incorporating a smoothing technique is negligible.
\end{remark}

Suppose the smooth component $f(x)$ in the objective function is not just convex but also strongly convex, the stochastic gradient algorithms developed in \cite{Lan:10b} and \cite{Hu:09} can achieve a convergence rate of $O(\frac{1}{N})$. Similar to the only convex cases, this convergence rate consists of two components, one term of $O(\frac{L}{N^2})$ which is not dominating but contains the Lipschitz constant $L$ and one term of  $O(\frac{1}{N})$ which is the bottle neck but independent of $L$. Hence, by the same reasons as above, if we incorporate the smooth technique into the algorithms in \cite{Lan:10b} and \cite{Hu:09} for strongly convex objective functions just as we did in Algorithm \ref{algo:sgsmoothappx}, we can obtain similar smoothing stochastic gradient algorithms with a convergence rate $O(\frac{1}{N})$.

\section{Numerical Results}
In this section, we apply our algorithms to four different types of regularized regression problems which belong to the class of problems formulated by (\ref{sobj}). We compare our numerical results with the AC-SA algorithm proposed by Ghadimi and Lan in \cite{Lan:10b}. We used a Matlab implementation and ran the experiments in a computer with an Intel(R) Core(TM)2 Duo CPU T8300 2.40GHz processor and 2.00GB RAM.

\subsection{Regularized Linear Regression with Discrete Probability Distribution}
Suppose there are $K$ data points $\{(x_i,y_i)\}_{i=1}^K$ with $x_i\in \mathbb{R}^p$ and $y_i\in \mathbb{R}$.
The task of \textsl{linear regression} is to find the parameters $\bb\in \mathbb{R}^p$ to fit the linear model $y=\bb^Tx+\epsilon$ by minimizing the average square loss function
\begin{equation*}
f_1(\bb)=\frac{1}{2}\sum_{i=1}^{K}\frac{\|x_i^T\bb-y_i\|^2}{K}=\frac{1}{2K}\|\bX\bb-\by\|^2,
\end{equation*}
where $\bX=[x_1,\dots,x_K]^T$ and $\by=[y_1,\dots,y_K]^T$. Here, we assume each instance $(x_i,y_i)$ occurs with equal chance, i.e., with a probability $\frac{1}{K}$ so that $f_1(\bb)$ is essentially the $\frac{1}{2}$ multiple of the expectation of the square loss $(x^T\bb-y)^2$.

The gradient of $f_1(\bb)$ is 
\begin{equation*}
 \nabla f_1(\bb)=\frac{1}{K}\bX^T(\bX\bb-\by).
\end{equation*}
It is easy to prove that $\nabla f_1(\bb)$ is Lipschitz continuous with a Lipschitz constant $L=\lambda_{\max}(\bX^T\bX)$ which denotes the largest eigenvalue of matrix $\bX^T\bX$.

Since we are testing stochastic gradient descent algorithms, we have to generate the stochastic gradient for $f_1(\bb)$ in each iteration. We first randomly sample a subset $\{(x_i,y_i)\}_{i\in S}$ with $S\subset\{1,2,\dots,K\}$ from the whole data set and the stochastic gradient  $G_1(\bb,S)$  corresponding to this sample is
\begin{equation}
\label{lsrsg}
 G_1(\bb,S)=\frac{1}{|S|}\bX_S^T(\bX_S\bb-\by_S),
\end{equation}
where $\bX_S$ and $\bY_S$ are sub-matrices of $\bX$ and $\bY$ whose rows are indexed by the elements of $S$.

When the data points belong to a high dimensional space, we are interested in selecting a small number of input features of the data which contribute most to influence the output. Hence, we want to minimize $f_1(\bb)$ with a regularization term $\Omega(\bb)$ which forces a highly sparse $\bb$ with zeros in the components corresponding to the less relevant input features. Then the regularized linear regression problem is defined as
\begin{equation}
\label{rlsr}
\min_{\bb}f_1(\bb)+\lambda\Omega(\bb),
\end{equation}
where $\lambda$ is the parameter that controls the regularization level.

In our numerical experiments, we consider two different choices of $\Omega(\bb)$. One choice is simply the $\ell_1$-norm of $\bb$, i.e., 
\begin{equation}
\label{1norm}
\Omega_1(\bb)=\|\bb\|_1.
\end{equation} 
A linear regression problem regularized by $\Omega_1(\bb)$ is also known as a lasso problem \cite{Tibshirani:96}.

We apply Algorithm 1 (SG) and Algorithm 2 (SSG) proposed in this paper and also the AC-SA algorithm proposed in \cite{Lan:10b} to problem (\ref{rlsr}) with $\Omega(\bb)=\Omega_1(\bb)$. In this case, the projection mappings in both SG and AC-SA have a closed form solution (see \cite{Liu:09}). In order to apply SSG, we observe that the non-smooth term in the objective function can be represented as $\lambda\Omega_1(\bb)=\max_{\|\ba\|_{\infty}\leq 1}\ba^TA\bb$ where $A=\lambda I$ and we choose $d(\ba)=\frac{1}{2}\|\ba\|^2$  as the strongly convex function in SSG.

We randomly generate a dataset $\{(x_i,y_i)\}_{i=1}^K$ as follows. First of all, we choose the real parameter $\hat{\bb}\in \mathbb{R}^p$ to be $[1,1,\dots,1,0,0,\dots,0]^T$ with first $p/2$ components equal to $1$ and last $p/2$ components equal to $0$.
And then, we generate each data point $x_i\in \mathbb{R}^p$ by generating each of its component $x_{ij}$ from a standard normal distribution $N(0,1)$ independently and we generate $y_i$ by setting $y_i=\bb^Tx_i+\epsilon_i/10$ with $\epsilon_i$ generated from a standard normal distribution $N(0,1)$.

We generate a set of data as above with $K=1000$ and $p=20$ and we set the parameters $\lambda=0.1$ and the total number of iterations $N=50000$. In each iteration, we randomly sample $10$ data points, i.e. $|S|=10$, to generate the stochastic gradient $G(\bb,S)$ by (\ref{lsrsg}). The numerical performances of these three algorithms are shown in the left figure in Figure 1. The horizontal line represents the CPU running time and the vertical line represents the value of objective function. 


Similarly, we apply these three algorithms to problem (\ref{rlsr}) with $\Omega(\bb)=\Omega_1(\bb)$ on a larger dataset with $K=100000$ and $p=200$. We still set $\lambda=0.1$ and $N=50000$ but we increase the sample size $|S|$ to $100$. The decreases of the objective values  with time by these three algorithms are shown in the right figure in Figure 1.
 
\begin{figure}
		\includegraphics[height=2in,width=2.5in]{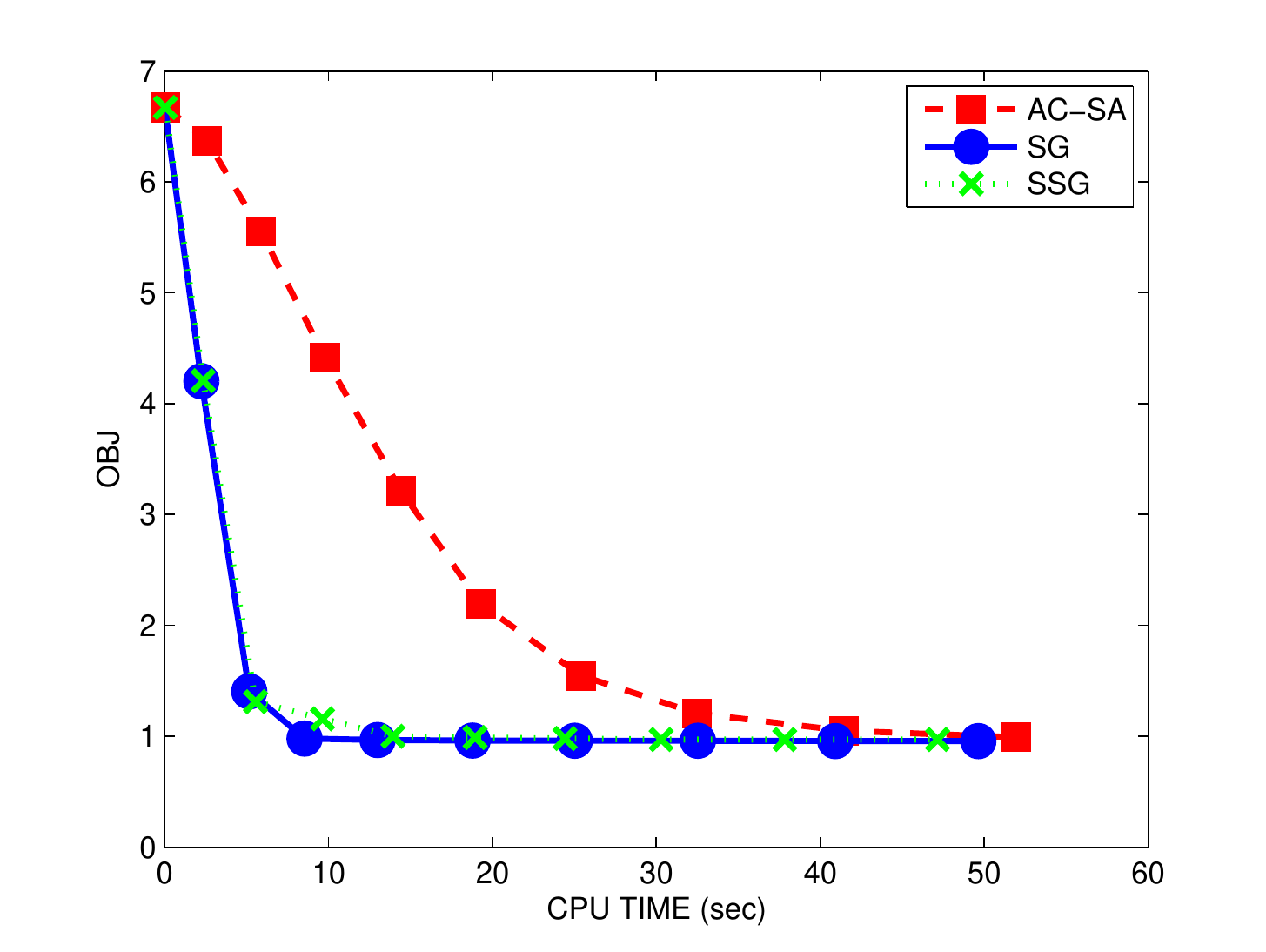}
		\includegraphics[height=2in,width=2.5in]{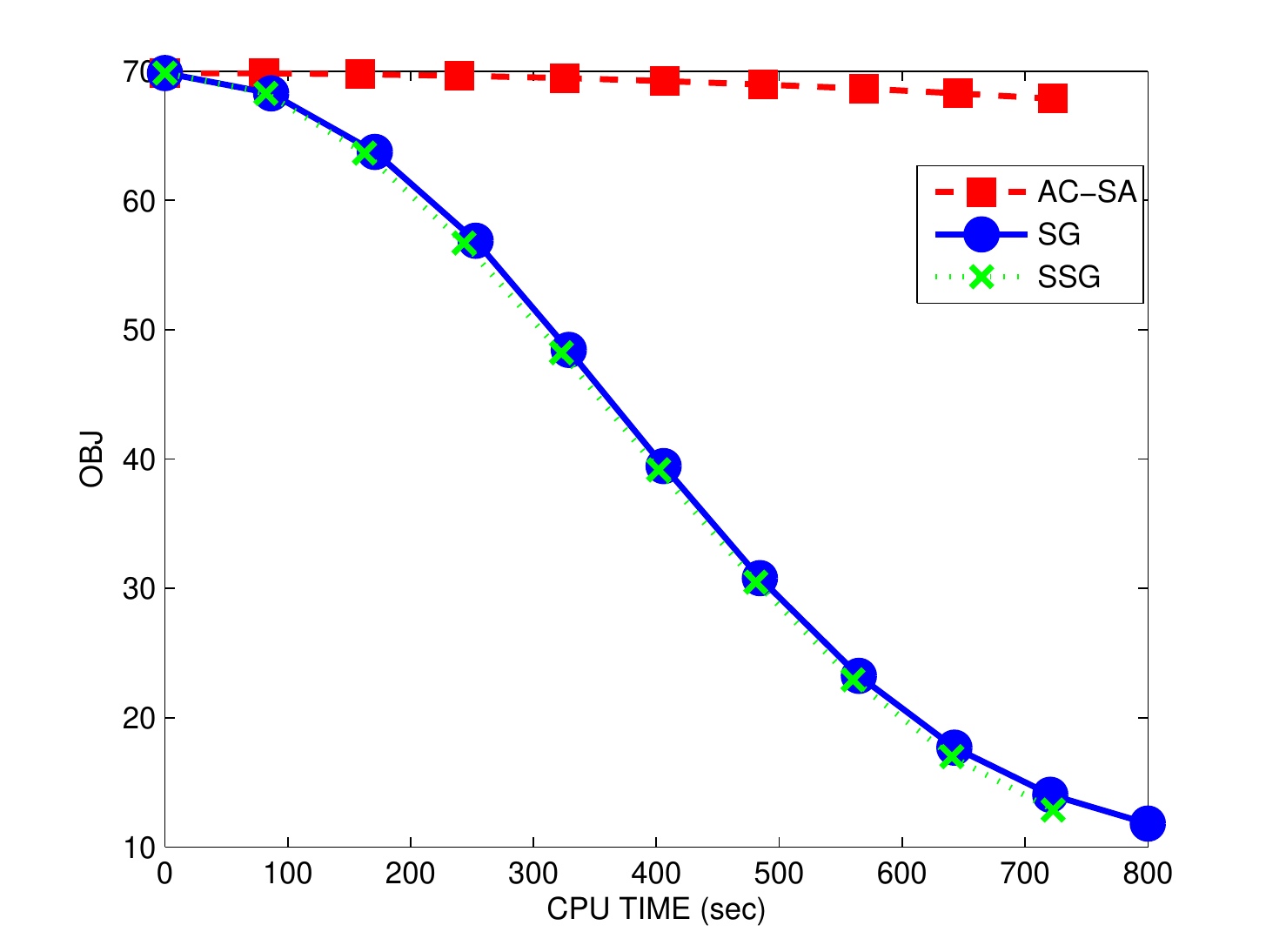}
	\caption{\centering linear regression with $\ell_1$-norm regularization \newline Left: $K=1000$, $p=20$; Right:$K=100000$, $p=200$.}
	\label{fig:lsr1n100020}
\end{figure}

The other choice for $\Omega(\bb)$ is the overlapped group sparsity inducing norm introduced by Jenatton et al. \cite{Bach:09}. Suppose the set of groups of inputs $\mathcal{G}=\{g_1,\dots,g_{|\mathcal{G}|}\}$ is a subset of the power set of $\{1,2,\dots,p\}$, the overlapped group sparsity inducing norm $\Omega_2(\bb)$ is defined as 
\begin{equation}
\label{groupnorm}
\Omega_2(\bb)=\sum_{g\in \mathcal{G}}w_g\|\bb_g\|,
\end{equation}
where $\bb_g\in \mathbb{R}^{|g|}$ is a sub-vector of $\bb$ which only contains the components of $\bb$ indexed by the elements of $g$ and $w_g$ is a positive constant for each $g\in \mathcal{G}$. 

To be specific, in our numerical experiments, we set $p=2^n$ for a positive integer $n$ and $w_g=\sqrt{|g|}$ and we define the set of groups of inputs $\mathcal{G}$ as follows
\begin{equation}
\label{treegroup}
\mathcal{G}=\left\{
\begin{array}{l}
g_{0,1},g_{0,2},\dots,\dots,g_{0,2^n},\\
g_{1,1},g_{1,2},\dots,g_{1,2^{n-1}},\\
\dots,\dots,\dots,\\
g_{i,1},\dots,g_{i,2^{n-i}},\\
\dots,\dots,\\
g_{n,1}\\
\end{array}
\right\},
\end{equation}
where
\begin{equation}
\label{treeg}
g_{i,j}=\{(j-1)2^i+1,(j-1)2^i+2,\dots,j2^i\}\mbox{ for } i=1,2,\dots,n \mbox{  and }j=1,2,\dots,2^{n-i}.
\end{equation}
This particular type of overlapped group sparsity inducing norm is also called \textsl{hierarchical norm}\cite{Bach:09}.

We apply algorithm SG, SSG and AC-SA to the problem (\ref{rlsr}) with $\Omega(\bb)=\Omega_2(\bb)$. In this case, the projection mappings in SG and AC-SA algorithms no longer have a closed form solutions. Jenatton et al. \cite{Jenatton:10} propose a coordinate descent method which can solve the projection mappings within $|\mathcal{G}|$ iterations. We adopt their method as a subroutine for solving the projection mappings when we apply SG and AC-SA to the hierarchical norm regularized regression problem. 

In order to apply SSG, we need to reformulate the non-smooth term $\lambda \Omega_2(\bb)$ with formulation (\ref{maxh}). Since the dual norm of Euclidean norm is Euclidean norm itself, 
$\|\bb_g\| =  \max_{\|\ba_g\| \leq 1} \ba_g^T\bb_g,$ where
$\ba_g \in \mathbb{R}^{|g|}$ is the vector of auxiliary variables
associated to $\bb_g$. Let $\ba=\left[\ba_{g_1}^T, \ldots,
\ba_{g_{|\mathcal{G}|}}^T\right]^T$ be the vector of length $\sum_{g
\in \mathcal{G}} |g|$ and denote the domain of $\ba$ by $\mathcal{Q} \equiv
\{\ba \ |\ \|\ba_g\| \leq 1 ,  \ \forall g \in \mathcal{G} \}$.
Note that, $\mathcal{Q}$ is the Cartesian product of unit balls in
Euclidean space which is a closed and convex set. We can rewrite $\lambda \Omega_2(\bb)$ as:
\begin{equation}
\label{eq:saddle_pen}
\lambda \Omega_2(\bb)= \lambda  \sum_{g \in \mathcal{G}}  w_g \max_{\|\ba_g\| \leq 1}
  \ba_g^T\bb_g =\max_{\ba \in \mathcal{Q}} \sum_{g \in \mathcal{G}}  \lambda w_g
  \ba_g^T\bb_g  =\max_{\ba \in \mathcal{Q}} \ba^T A \bb,
\end{equation}
where $A \in \mathbb{R}^{{\sum_{g \in \mathcal{G}}|g|} \times J}$ is
a matrix such that $A\bb=\left[\lambda w_{g_1}\bb_{g_1}^T, \ldots,
\lambda w_{g_{|\mathcal{G}|}}\bb_{g_{|\mathcal{G}|}}^T\right]^T$. The rows of $A$ are indexed by all pairs of $(i,g)$ such that $i\in \{1,\ldots, p\},  i\in g$ and its columns are indexed by $j \in \{1,
\ldots, p\}$ and $A$ is defined as:
\begin{equation}
\label{eq:matrixC} A_{(i,g),j}=\left\{
\begin{array}{ll}
\lambda w_g&\mbox{if }i=j\\
0&\mbox{otherwise}
\end{array}.
\right.
\end{equation}
Different from SG and AC-SA, the projection mapping in AC-SA always has a closed form solution.

We generate a dataset $\{(x_i,y_i)\}_{i=1}^K$ in the same way as before with $K=1000$ and $n=5$ ($p=2^n=32$) and set the parameters $\lambda=0.1$, the total number of iterations $N=10000$ and the sample size $|S|=10$. The numerical results by these algorithms on this data set are shown in the left figure in Figure 2. Also, we generate a larger data set with $K=100000$ and $n=9$ ($p=2^n=512$) and run the algorithms on it with  $\lambda=0.1$, $N=10000$ and $|S|=100$. The numerical results are posted in the right figure in Figure 2.

\begin{figure}
		\includegraphics[height=2in,width=2.5in]{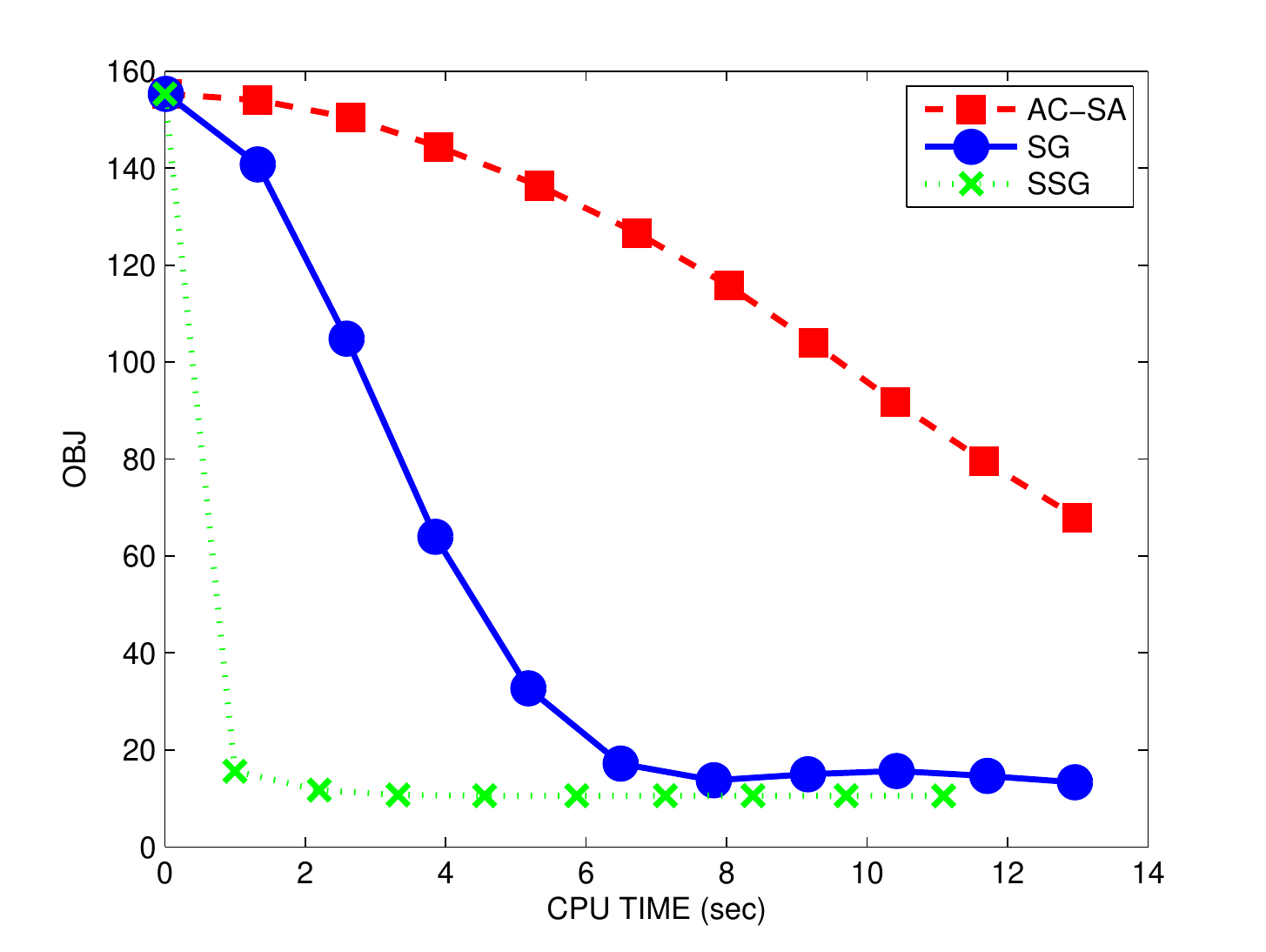}
		\includegraphics[height=2in,width=2.5in]{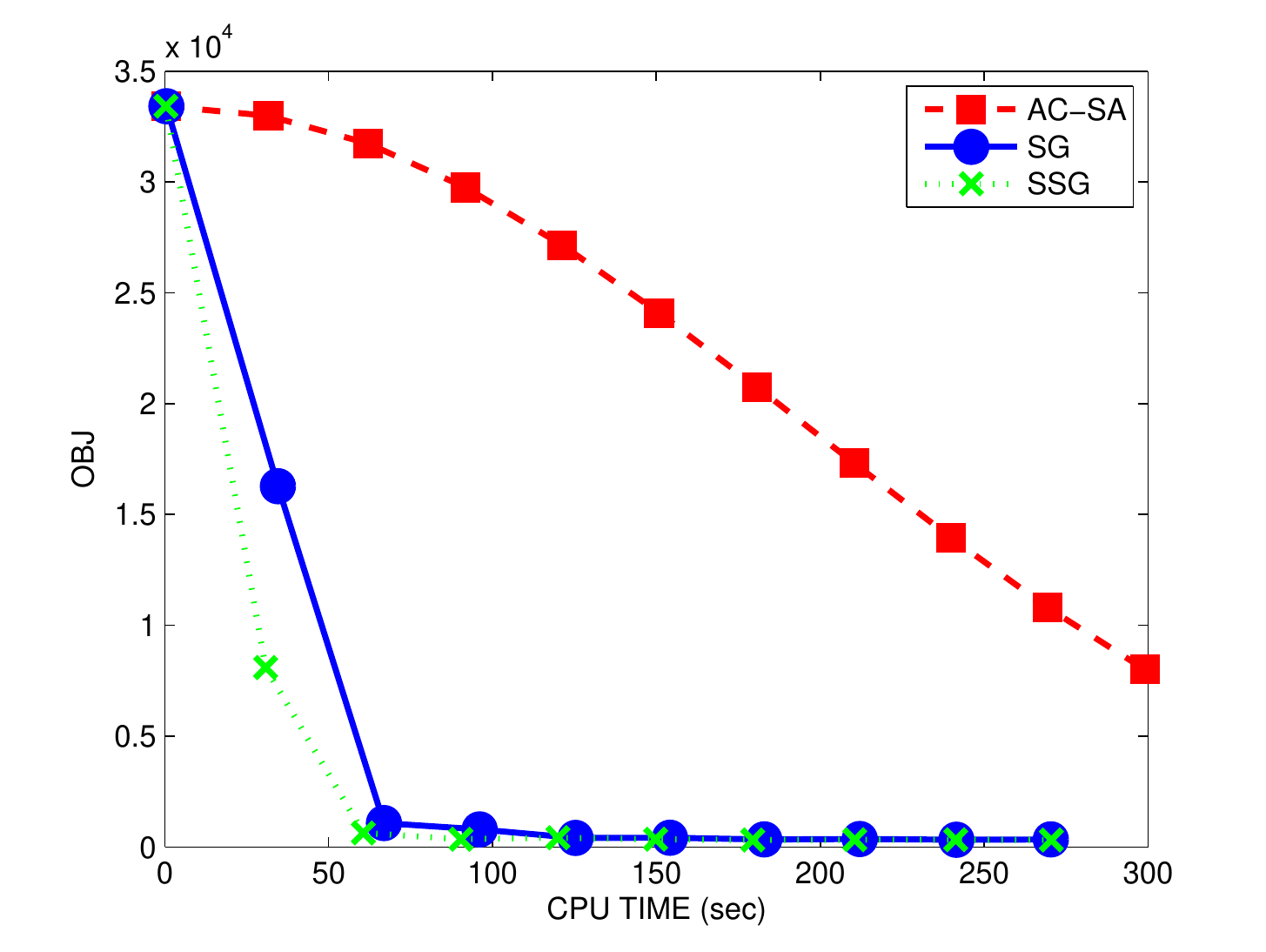}
\caption{\centering linear regression with hierarchical norm regularization \newline Left: $K=1000$, $p=32$; Right:$K=100000$, $p=512$.}
	\label{fig:lsrtree100032}
\end{figure}

From Figure 1 and Figure 2, we can imply that even though the SG, SSG and AC-SA have the same $O(\frac{1}{\sqrt{N}})$ theoretical convergence rate, their performances are different in practical applications. In three out of the four experiments, SG and SSG are more efficient than AC-SA. One of the reasons for this result is that AC-SA algorithm chooses a very small step length (see the choice of $\gamma_t$ in Proposition 4 in \cite{Lan:10b}) in order to mitigate the impact from the inaccuracy of the stochastic gradient. This is needed in the theoretical proof of the convergence rate of AC-SA. However, when the stochastic gradient is a good approximation for the exact gradient, e.g., $\sigma^2$ is very small, the small step length in AC-SA is too conservative. Instead, SG and SSG adopt a relatively larger step length such that they can reduce the objective functions more efficiently.

The influence of the smoothing technique by Nesterov \cite{Nesterov:05} is also reflected by these numerical results. When $\Omega_1(\bb)$ is chosen as the regularization term, the projection mapping in SG has a closed form solution so that applying the smoothing technique is not necessary. Hence, in Figure 1, the blue curve (SG) and the green curve (SSG) almost overlap. This is comply with that fact that SG and SSG have a same $O(\frac{1}{\sqrt{N}})$ complexity shown by Theorem \ref{smoothconveg} and Theorem \ref{nonsmoothconveg}.

However, in Figure 2, we can see that SSG is much more efficient than SG. This is because SG has to use a coordinate descent method to solve the projection mapping due to the lack of closed form solution when $\Omega_2(\bb)$ is the regularization term. Even though the coordinate descent method is shown to converge after finite steps. It is still not necessarily faster than solving it by a closed form which is available in SSG.

\subsection{Regularized Linear Regression with Continuous Probability Distribution}
Here, we apply our algorithms again on the regularized linear regression problems. However, this time, we assume that there are infinitely many data points $(x,y)\in \mathbb{R}^{p+1}$ which follow a continuous distribution $p(x,y)$. The task is still to find the parameters $\bb\in \mathbb{R}^{p}$ to fit the linear model $y=\bb^Tx+\epsilon$ by minimizing the average square loss function
\begin{equation*}
f_1^c(\bb)=\frac{1}{2}\mathbb{E}(x^T\bb-y)^2=\frac{1}{2}\int{(x^T\bb-y)^2}p(x,y)dxdy.
\end{equation*}

In our numerical experiment, we make $x$ follow the standard normal distribution in $\mathbb{R}^{p}$, i.e, $p(x)=N(0,I)$. The real parameters $\hat{\bb}$ are chosen to be $[1,1,\dots,1,0,0,\dots,0]^T$ with first $p/2$ components equal to $1$ and last $p/2$ components equal to $0$ and the error term $\epsilon$ in the linear model is assumed to have a standard normal distribution $N(0,1)$ so that the variable $y$ follows a normal distribution $p(x|y)=N(x^T\hat{\bb},1)$ once $x$ is fixed. By these settings, the distribution $p(x,y)$ in our numerical experiments is
\begin{equation*}
p(x,y)=p(x)p(y|x)=\frac{1}{(2\pi)^{\frac{p}{2}}}e^{-\frac{1}{2}x^Tx}\frac{1}{(2\pi)^{\frac{1}{2}}}e^{-\frac{1}{2}(y-\hat{\bb}^Tx)^2}=\frac{1}{(2\pi)^{\frac{p+1}{2}}}e^{-\frac{x^Tx+(y-\hat{\bb}^Tx)^2}{2}}.
\end{equation*}
It is easy to show that, in this case, the loss function $f_1^c(\bb)$ becomes
\begin{equation*}
f_1^c(\bb)=\frac{1}{2}(\bb^T\bb-2\bb^T\hat{\bb}+\frac{p}{2}+1),
\end{equation*}
whose gradient is simply $\nabla f_1^c(\bb)=\bb-\hat{\bb}$ with a Lipschitz constant $L=1$.

Similar to the discrete cases, we apply our algorithms to the following regularized linear regression problem 
\begin{equation}
\label{rlsrcon}
\min_{\bb}f_1^c(\bb)+\lambda\Omega(\bb),
\end{equation}
where $\Omega(\bb)$ is the regularization term.

In order to generate a stochastic approximation for $\nabla f_1^c(\bb)$, in each iteration, we sample a set of points $S=\{(x_i,y_i)\}_{i=1,\dots,|S|}$ by generating $x_i$ from $N(0,I)$ and $\epsilon_i$ from $N(0,1)$ and setting $y_i=x_i^T\hat{\bb}+\epsilon_i$ for  $i=1,\dots,|S|$. Then we can compute a stochastic gradient $G(\bb,S)$ by (\ref{lsrsg}).

First, we apply AC-SA, SG and SSG algorithms on (\ref{rlsrcon}) with $\Omega(\bb)=\Omega_1(\bb)$, $p=1000$, $|S|=10$ and $\lambda=0.1$. The performances of these algorithms are shown in Figure \ref{fig:lsrcon1n1000}.

\begin{figure}
	\centering
		\includegraphics[height=2.5in,width=3in]{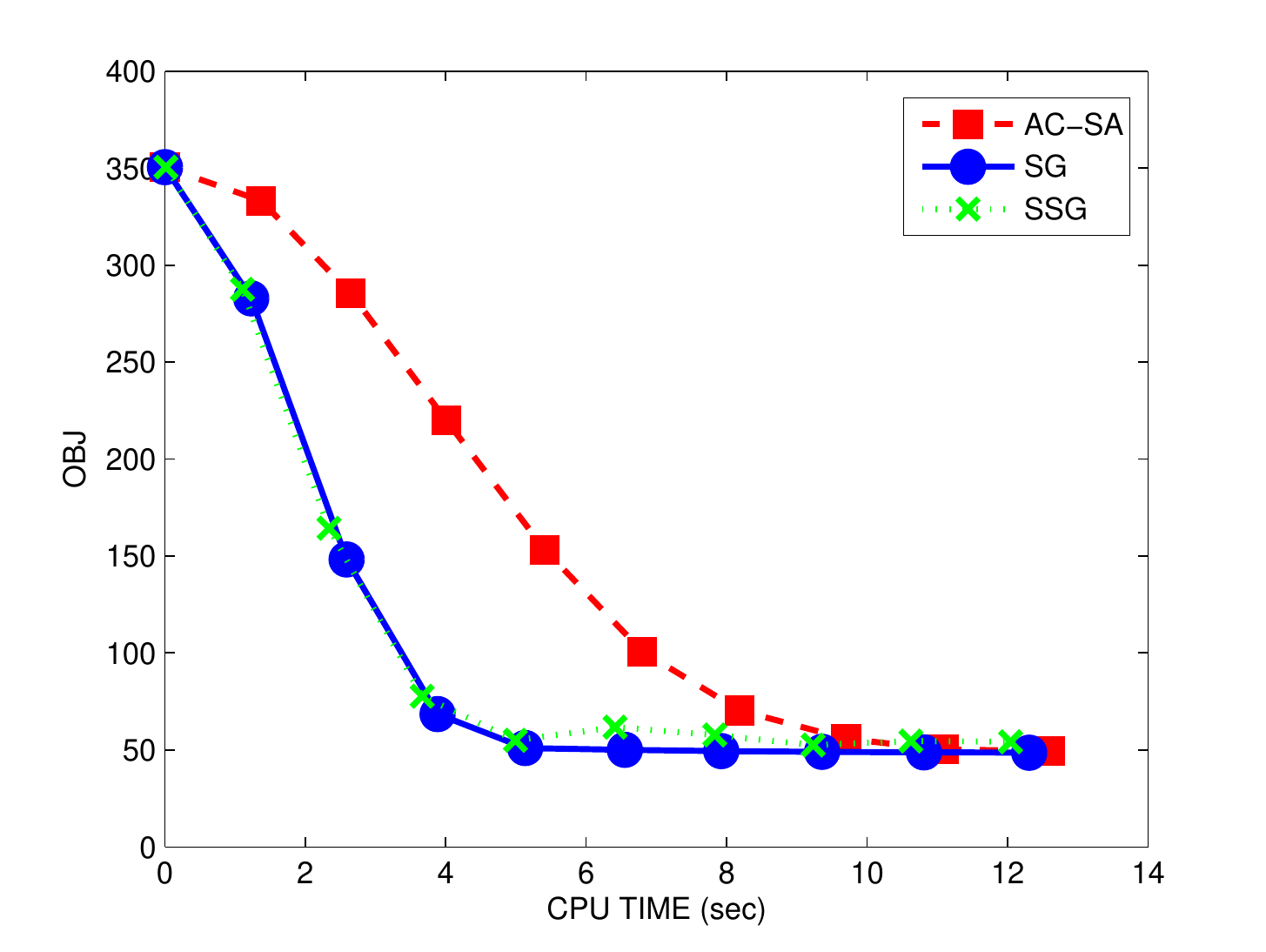}
		\caption{\centering linear regression with $\ell_1$-norm regularization and $K=+\infty$, $p=1000$}
	\label{fig:lsrcon1n1000}
\end{figure}

Then, we apply these three algorithms on on (\ref{rlsrcon}) with $\Omega(\bb)=\Omega_2(\bb)$, $n=8 (p=2^n=256)$, $|S|=100$ and $\lambda=0.1$. The numerical performances are presented in Figure \ref{fig:lsrcontree256}.

\begin{figure}
	\centering
		\includegraphics[height=2.5in,width=3in]{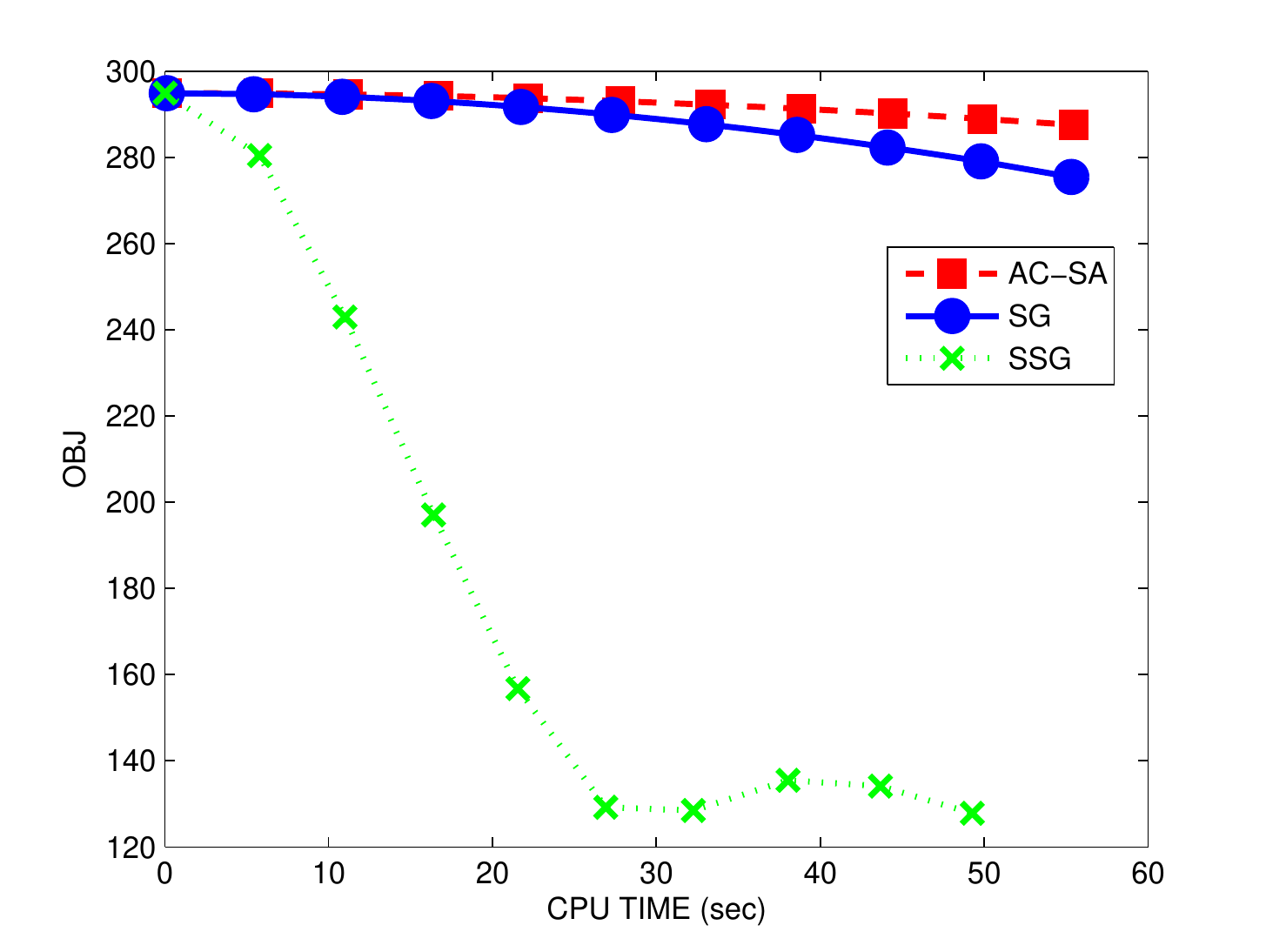}
		\caption{\centering linear regression with hierarchical norm regularization and $K=+\infty$, $p=256$}
	\label{fig:lsrcontree256}
\end{figure}

Figure \ref{fig:lsrcon1n1000} and Figure \ref{fig:lsrcontree256} reflect the similar phenomenons as Figure \ref{fig:lsr1n100020} and Figure \ref{fig:lsrtree100032}. SG and SSG converge faster than AC-SA and when the regularization term is complicated, SSG significantly outperforms the other two algorithms.


\subsection{Regularized Logistic Regression}
Suppose there are $K$ data points $\{(x_i,y_i)\}_{i=1}^K$, where each $x_i\in \mathbb{R}^p$ is the predictor with its Euclidean norm $\|x_i\|=1$ and  $y_i\in\{0,1\}$ is the class label of $x_i$ which indicates that $x_i$ belongs to class $0$ or class $1$. We assume that the posterior probability of the class label of a particular predictor $x$ is given by
\begin{eqnarray}
\label{labprob}
\mbox{Pr}(y=1|x)=\frac{e^{\bb^Tx}}{1+e^{\bb^Tx}}\\
\mbox{Pr}(y=0|x)=\frac{1}{1+e^{\bb^Tx}}
\end{eqnarray}
for some $\bb\in\mathbb{R}^p$. The task of \textsl{logistic regression}  is to find the parameters $\bb\in \mathbb{R}^p$ by minimizing the minus log-likelihood corresponding to this set of data points, which is defined as
\begin{equation*}
f_2(\bb)=-\frac{1}{K}\sum_{i=1}^K\log(\mbox{Pr}(y_i|x_i))=\frac{1}{K}\sum_{i=1}^K\left[\log(1+e^{\bb^Tx_i})-y_i\bb^Tx_i\right].
\end{equation*}

Similar to the regularized linear regression problem, we minimize $f_2(\bb)$ together with a regularization term $\Omega(\bb)$ in order to obtain a sparse solution $\bb$. Hence, the regularized logistic regression can be formulated as
\begin{equation}
\label{rlr}
\min_{\bb}f_2(\bb)+\lambda\Omega(\bb),
\end{equation}
where $\Omega(\bb)$ can also be chosen to be $\Omega_1(\bb)$ or the hierarchical norm $\Omega_2(\bb)$  defined by(\ref{groupnorm}), (\ref{treegroup}) and (\ref{treeg}).

The gradient of $f_2(\bb)$ is the following
\begin{equation*}
    \nabla f_2(\bb)=\frac{1}{K}
\sum_{i=1}^K
\left(\frac{e^{\bb^Tx_i}}{1+e^{\bb^Tx_i}}-y_i\right) x_i.
\end{equation*}
Because each data point satisfies $\|x_i\|=1$, it can be shown that $\nabla f_2(\bb)$ is Lipschitz continuous with a Lipschitz constant $L=1$. Similar to the regularized linear regression problems,  we randomly sample a subset $\{(x_i,y_i)\}_{i\in S}$ with $S\subset\{1,2,\dots,K\}$ from the whole data set and generate the stochastic gradient  $G_2(\bb,S)$  for $f_2(\bb)$ as follows
\begin{equation}
\label{lrsg}
 G_2(\bb,S)=\frac{1}{|S|}
\sum_{i\in S}
\left(\frac{e^{\bb^Tx_i}}{1+e^{\bb^Tx_i}}-y_i\right) x_i.
\end{equation}

Now we apply SG, SSG and AC-SA to problem (\ref{rlr}) with $\Omega(\bb)=\Omega_1(\bb)$. We create a set of artificial data $\{(x_i,y_i)\}_{i=1}^K$ with $K=1000$ and $p=20$ as follows. At first, we choose the real parameter $\hat{\bb}$ to be an all-ones vector in $\mathbb{R}^p$. After that, for each $i=1,\dots,K$, we create a $\hat{x}_i$ by generating each of its component $\hat{x}_{ij}$ from a standard normal distribution $N(0,1)$ independently and we get $x_i$ by normalizing $\hat{x}_i$, i.e., $x_i=\hat{x}_i/\|\hat{x}_i\|$. The corresponding $y_i$ is set to be $1$ or $0$ randomly with the probabilities defined by (\ref{labprob}). We set the sample size $|S|=10$, $\lambda=0.01$ and the number of iterations $N=50000$ in all of the three algorithms. The numerical performances are presented in Figure \ref{fig:lr1n100020}.

\begin{figure}
	\centering
		\includegraphics[height=2.5in,width=3in]{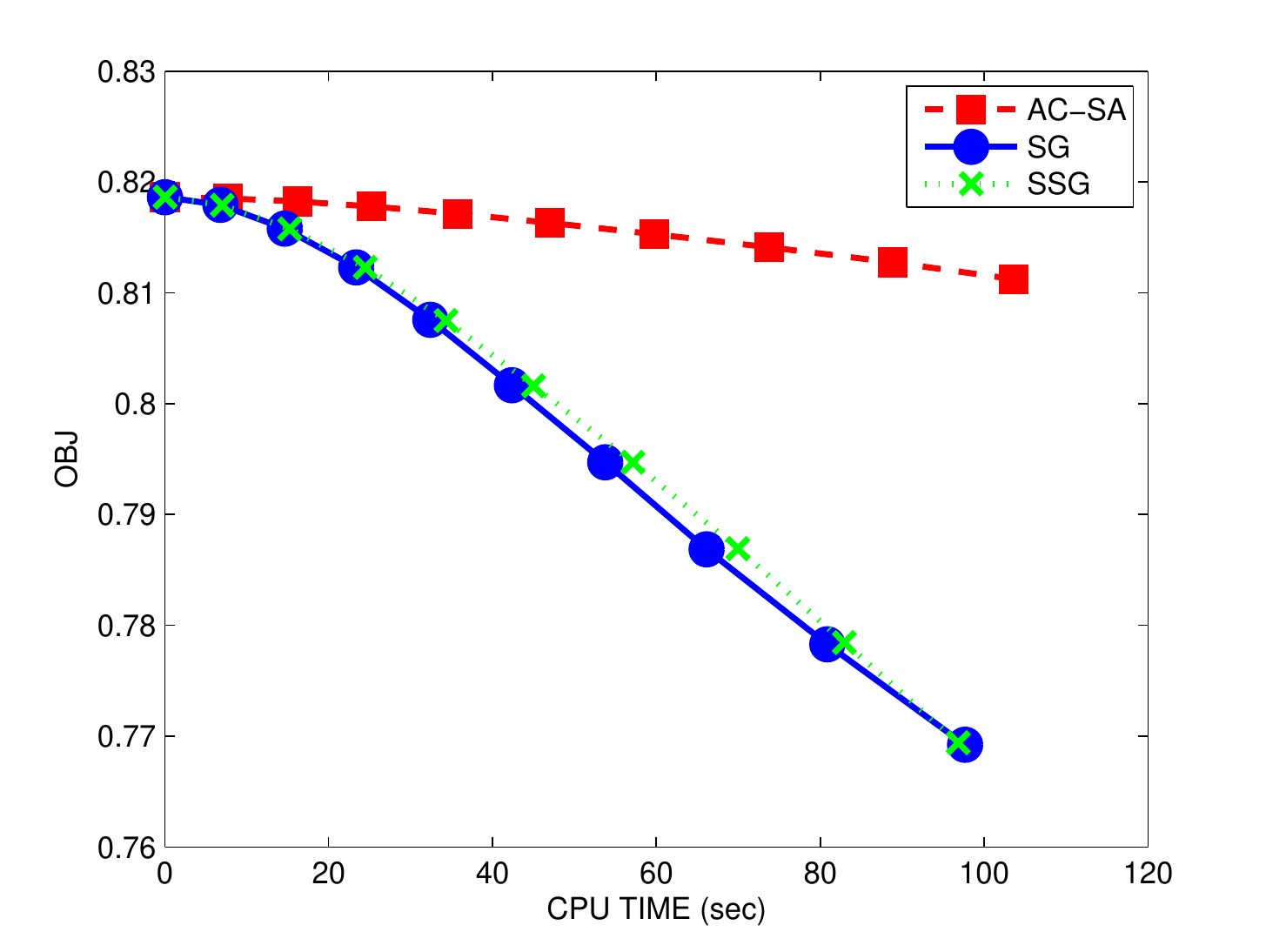}
		\caption{\centering logistic regression with $\ell_1$-norm regularization and $K=1000$, $p=20$}
	\label{fig:lr1n100020}
\end{figure}

For the problem (\ref{rlr}) with $\Omega(\bb)=\Omega_2(\bb)$, we generate the data in the same way as above but with $K=1000$ and $n=5$ ( $p=32$ ). Still, $|S|$, $\lambda$ and $N$ are set to be $10$, $0.01$ and $50000$ respectively. We put the curves in Figure \ref{fig:lrtree100032} to show how the objective values decrease in these algorithms.

\begin{figure}
	\centering
		\includegraphics[height=2.5in,width=3in]{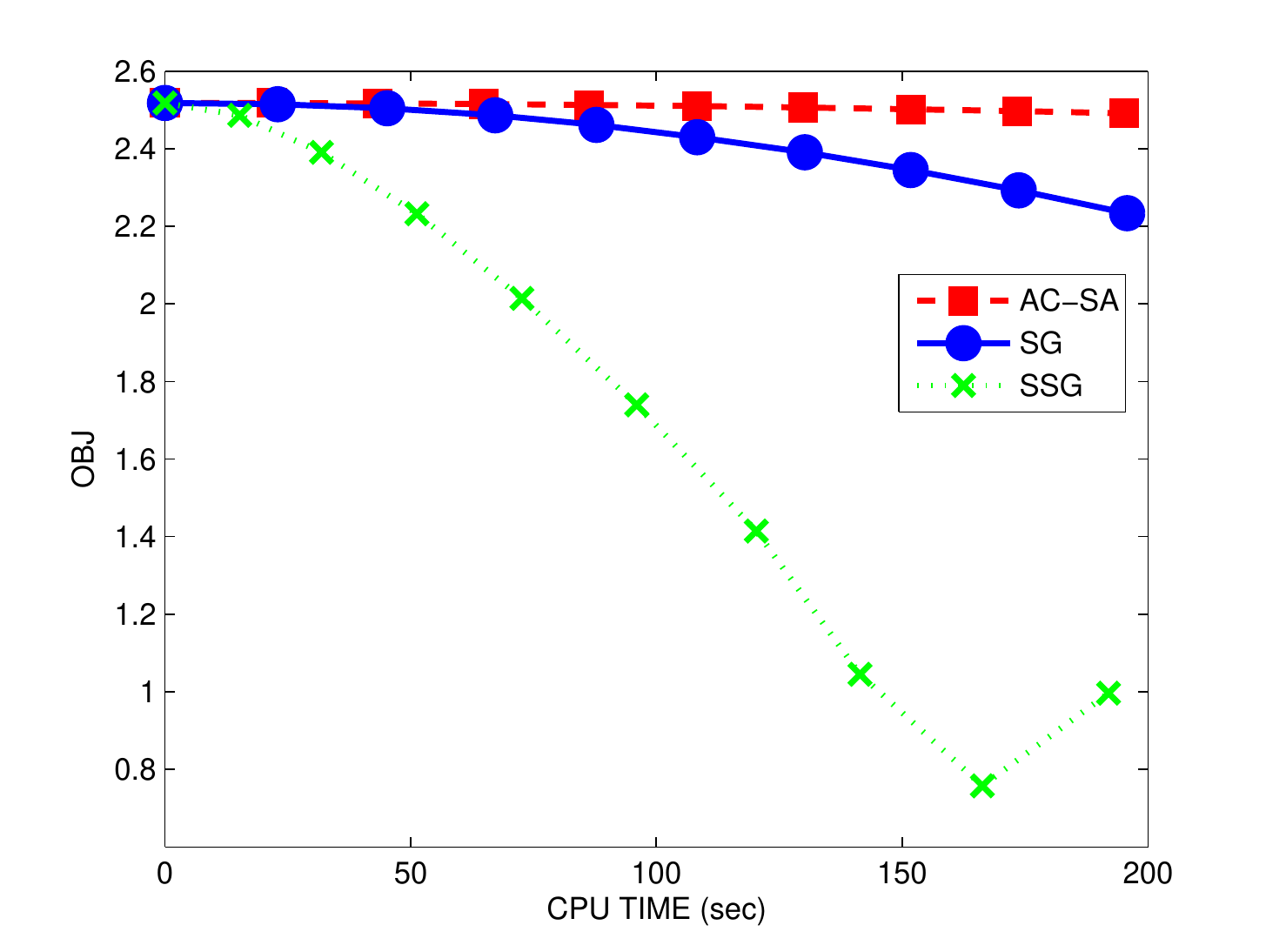}
		\caption{\centering logistic regression with hierarchical norm regularization and $K=1000$, $p=32$}
	\label{fig:lrtree100032}
\end{figure}

The properties of our algorithms shown by Figure \ref{fig:lr1n100020} and Figure \ref{fig:lrtree100032} are very similar to what are shown in Figure 1,2,3,4. In the $\ell_1$-norm regularized logistic regression problems, SG and SSG are more efficient than AC-SA due to the more aggressive choices of the step lengths. In the cases of hierarchical norm regularized logistic regression, SSG is  more efficient than the other two just because SSG has a closed form solution for its projection mapping but SG and AC-SA have to rely on another algorithm as a subroutine to solve their projection mappings.

\section{Conclusion}
In this paper, we consider an optimization problem whose objective function is a composition of a smooth convex function and a non-smooth convex function. We first developed a stochastic gradient descent algorithm for solving this problem. We also proposed another stochastic gradient descent algorithm by smoothing the non-smooth term in the objective function. The convergence rates of these two algorithm are proved.  The results of our numerical experiments demonstrate efficiency and scalability of our algorithms.

\bibliography{summerpaperref}

\bibliographystyle{plain}

\end{document}